\documentclass[12pt]{amsart}

\usepackage{amssymb, hyperref,xcolor, mathrsfs, cancel, fullpage, yfonts}
\usepackage[capitalise]{cleveref}

\newtheorem{theorem}{Theorem}[section]
\newtheorem{lemma}[theorem]{Lemma}
\newtheorem{condition}[theorem]{Condition}

\newtheorem{corollary}[theorem]{Corollary}

\theoremstyle{remark}
\newtheorem{remark}[theorem]{Remark}

\numberwithin{equation}{section}

\newcommand{\ecc}{\ensuremath{\mathfrak{h}}}
\newcommand{\vcc}{\ensuremath{h}}
\newcommand{\edge}{\ensuremath{\mathrm{edge}}}
\newcommand{\ver}{\ensuremath{\mathrm{vert}}}
\newcommand{\scrV}{\ensuremath{\mathscr{V}}}
\newcommand{\calG}{\ensuremath{\mathcal{G}}}
\newcommand{\calO}{\ensuremath{\mathcal{O}}}
\newcommand{\mo}{{-1}}
\newcommand{\scrP}{\ensuremath{\mathscr{P}}}
\newcommand{\calN}{\ensuremath{\mathcal{N}}}
\newcommand{\xra}{\xrightarrow}
\newcommand{\bbC}{\ensuremath{\mathbb{C}}}
\newcommand{\vol}{\ensuremath{\mathrm{vol}}}
\newcommand{\calX}{\ensuremath{\mathcal {X}}}
\newcommand{\calI}{\ensuremath{\mathcal I}}

\title{A Cheeger inequality for the lower spectral gap}

\author{Jyoti Prakash Saha}
\address{Department of Mathematics, Indian Institute of Science Education and Research Bhopal, Bhopal Bypass Road, Bhauri, Bhopal 462066, Madhya Pradesh,
India}
\curraddr{}
\email{jpsaha@iiserb.ac.in}
\thanks{}


\begin{document}

\maketitle

\begin{abstract}
Let $\Gamma$ be a Cayley graph, or a Cayley sum graph, or a twisted Cayley graph, or a twisted Cayley sum graph, or a vertex-transitive graph. Denote the degree of $\Gamma$ by $d$, its edge Cheeger constant by $\mathfrak{h}_\Gamma$, and its vertex Cheeger constant by $\vcc_\Gamma$. Assume that $\Gamma$ is undirected, non-bipartite. We prove that the edge bipartiteness constant of $\Gamma$ is $\Omega({\mathfrak{h}_\Gamma}/{d})$, the vertex bipartiteness constant of $\Gamma$ is $\Omega(\vcc_\Gamma)$, and the smallest eigenvalue of the normalized adjacency operator of $\Gamma$ is $-1 + \Omega({\vcc_\Gamma^2}/{d^2})$. This answers in the affirmative a question of Moorman, Ralli and Tetali on the lower spectral gap of Cayley sum graphs. 
\end{abstract}

\section{Introduction}

In the following, $\Gamma$ denotes a finite, undirected, non-bipartite, $d$-regular graph on a set of $n$ vertices. 
The edge Cheeger constant (resp. the vertex Cheeger constant, the edge bipartiteness constant) of $\Gamma$ is denoted by $\ecc_\Gamma$ (resp. $\vcc_\Gamma, \beta_{\edge, \Gamma}$). 
The gap between the largest eigenvalue (which is equal to $1$) and the second largest eigenvalue (denoted by $\mu_2$) of the normalized adjacency operator of $\Gamma$ is called the \emph{spectral gap} of $\Gamma$. 
The gap between $-1$ and the smallest eigenvalue (denoted by $\mu_n$) of the normalized adjacency operator of $\Gamma$ is called the \emph{lower spectral gap} of $\Gamma$.

The discrete Cheeger--Buser inequality, established by Dodziuk \cite{DodziukDifferenceEqnIsoperimetricIneq}, Alon and Milman \cite{AlonMilmanIsoperiIneqSupConcen}, and Alon \cite{AlonEigenvalueExpand}, states that the edge Cheeger constant of a finite graph $\Gamma$ and its spectral gap satisfy 
$$
\frac{\ecc_\Gamma^2}2 \leq 1 - \mu_2 \leq 
2 \ecc_\Gamma.$$
The problem of obtaining bounds for the lower spectral gap of $\Gamma$ is of considerable interest. 
Trevisan \cite{TrevisanMaxCutTheSmallestEigenVal} established an inequality, which states that the edge bipartiteness constant of a finite graph $\Gamma$ and its lower spectral gap satisfy 
$$
\frac{\beta_{\edge, \Gamma}^2}2
\leq 
1 + \mu_n
\leq 2\beta_{\edge, \Gamma}.
$$
The bound $1 + \mu_n \geq 
\frac{\beta_{\edge, \Gamma}^2}2$, due to Trevisan \cite[Equation (9)]{TrevisanMaxCutTheSmallestEigenVal}, was also obtained by Bauer and Jost 
\cite[p. 804--805]{BauerJostBipartiteNbdGraphsSpecLaplaceOp} using a technique due to Desai and Rao \cite{DesaiRaoCharacSmallestEigenvalue}. 

The lower spectral gap of graphs have been studied by several authors. 
Breuillard, Green, Guralnick and Tao showed  that for Cayley graphs, combinatorial expansion implies two-sided spectral expansion \cite[Appendix E]{BGGTExpansionSimpleLie}. 
Knox and Mohar proved that the lower spectral gap of an undirected, regular, non-bipartite, distance-regular graph with odd girth $g$ is greater than or equal to $1 - \cos \frac \pi g$ 
\cite[Corollary 5]{KnoxMoharFractionalDecomSmallestEigenVal}, which generalizes and strengthens a result of Qiao, Jing and Koolen \cite[Theorem 1]{QiaoJingKoolenNonBipartiteDistRegular}.

Moorman, Ralli and Tetali obtained a bound on the smallest non-trivial eigenvalue of the normalized adjacency operator of a Cayley graph. 
They proved that the vertex bipartiteness constant $\beta_{\ver, \Gamma}$ of a Cayley graph $\Gamma$ is $\Omega({\vcc_\Gamma})$, and applied Trevisan's inequality to establish that 
the lower spectral gap of the normalized adjacency operator of $\Gamma$ is $-1 + \Omega(\frac{\vcc_\Gamma^2}{d^2})$ \cite[Theorem 2.6]{CayleyBottomBipartite}. 
This improves the prior bound due to Biswas \cite{BiswasCheegerCayley}. 
They also pointed out that having a bound of the form $\beta_{\edge, \Gamma} = \Omega(\frac{\ecc_\Gamma}d)$ for Cayley graphs would be interesting, and established  an upper bound on $\ecc_\Gamma$ in terms of $\beta_{\edge, \Gamma}$ and $d$ for such graphs \cite[Theorem 2.5]{CayleyBottomBipartite}, which yields the bound $\beta_{\edge, \Gamma} = \Omega(\frac{\ecc_\Gamma}d)$ for Cayley graphs. 

Further, Moorman, Ralli and Tetali asked if the lower spectral gap of a non-bipartite Cayley sum graph admits a lower bound similar to the one for the lower spectral gap of Cayley graphs obtained by them in \cite[Theorem 2.6]{CayleyBottomBipartite}.  
In this direction, we have \cref{Thm:IntroCayleySum}. As a consequence, we deduce \cref{Cor:Bdd}, which improves the bounds obtained by Biswas and the author \cite{CheegerCayleySum,CheegerTwisted,VertexTra}. 

\begin{theorem}
\label{Thm:IntroCayleySum}
Let $\Gamma$ be a Cayley sum graph, or a twisted Cayley graph, or a twisted Cayley sum graph, or a vertex-transitive graph. Denote the degree of $\Gamma$ by $d$. Assume that $\Gamma$ is undirected, non-bipartite and has finitely many vertices.
Then, the inequalities  
$
\beta_{\edge, \Gamma }
\geq
\frac
{\ecc_\Gamma}
{90 d}
$, 
$
\beta_{\ver, \Gamma }
\geq
\frac
{\vcc_\Gamma}
{135}
$
hold. 
\end{theorem}

Applying Trevisan's inequality $1 + \mu_n \geq 
\frac{\beta_{\edge, \Gamma}^2}2$ and the bound 
$\beta_{\edge, \Gamma} \geq \frac{\beta_{\ver, \Gamma}}d$ \cite[Theorem 2.3]{CayleyBottomBipartite}, we obtain the following corollary of the above result. 

\begin{corollary}
\label{Cor:Bdd}
Let $\Gamma$ denote a graph as in 
\cref{Thm:IntroCayleySum}. Then its lower spectral gap is $ \Omega(\frac{\vcc_\Gamma^2}{d^2})$.
\end{corollary}

The bound on the lower spectral gap is obtained by using the bound on the bipartiteness constant and applying Trevisan's inequality. 
Since $\vcc_\Gamma \geq \ecc_\Gamma$ holds, it follows that the spectral gap of $\Gamma$ is $-1 + \Omega(\frac{\ecc_\Gamma^2}{d^2})$ for $\Gamma$ as in \cref{Thm:IntroCayleySum}. 
Note that upto a quadratic factor of the degree, this lower bound on the lower spectral gap  is analogous to the lower bound on the spectral gap provided by the discrete Cheeger--Buser inequality.  

For a non-bipartite, vertex-transitive graph $\Gamma$ of degree $d$, it was shown in \cite[Theorem 3.5.(3)]{VertexTra} that 
$
\vcc_{\Gamma^2} = 
\Omega
(
\frac
{\vcc_\Gamma^2}
{d^3}
)
$. 
If $\Gamma$ is a Cayley graph, or a Cayley sum graph, or a twisted Cayley graph, or a twisted Cayley sum graph of degree $d$ and $\Gamma$ is non-bipartite, then \cite[Theorem 3.5.(2)]{VertexTra} states that 
$
\vcc_{\Gamma^2} = 
\Omega
(
\frac
{\vcc_\Gamma^2}
{d}
)
$. 
The following result provides an analogue of \cite[Theorem 3.5]{VertexTra}. 
The bound for $\ecc_{\Gamma^2}$ in terms of $\ecc_\Gamma$ and $d$ provided by \cref{Thm:EccEccBddAlgGraph} 
is stronger than the ones that can be obtained from \cite[Theorem 2.6]{CayleyBottomBipartite} for Cayley graphs, and those provided by \cite[Theorem 3.5]{VertexTra}.

\begin{theorem}
\label{Thm:EccEccBddAlgGraph}
Let $\Gamma$ be a Cayley graph, or a Cayley sum graph, or a twisted Cayley graph, or a twisted Cayley sum graph, or a vertex-transitive graph. Denote the degree of $\Gamma$ by $d$. Assume that $\Gamma$ is undirected, non-bipartite and has finitely many vertices. Then 
the inequality 
$\ecc_{\Gamma^2} 
> 
\frac
{\ecc_\Gamma^2}
{48 d}$ 
holds. 
\end{theorem}

\cref{Thm:IntroCayleySum}
follows from Theorems \ref{Thm:betaEdgeEcc}, \ref{Thm:betaEdgeVcc}. 
\cref{Thm:EccEccBddAlgGraph}
is established using \cref{Thm:Sec2Intro} at the end of \cref{Sec:CheegerConstantOfSquareGraph}.

\begin{remark}
Soon after a preprint of this work was posted on arXiv, Hu and Liu established stronger bounds on the lower spectral gap of Cayley graphs, Cayley sum graphs, and vertex-transitive graphs which are undirected and non-bipartite \cite{HuLiuVertexIsoperimetrySignedGraph}. 
First, they proved a result \cite[Theorem 1.2]{HuLiuVertexIsoperimetrySignedGraph} extending a work of 
Bobkov, Houdr\'{e} and Tetali \cite{BobkovHoudreTetalisLambdaInfVertexIsoperimetry}. 
This implies that for a $d$-regular graph $\Gamma$, the lower spectral gap is $\Omega(\frac{\beta_{\ver, \Gamma}^2}{16d})$ \cite[Equation (1.7)]{HuLiuVertexIsoperimetrySignedGraph}.
Next, they showed that the non-bipartite Cayley sum graphs and the non-bipartite  vertex-transitive graphs satisfy a bound of the form $\beta_{\ver, \Gamma} = \Omega(\vcc_\Gamma)$ \cite[Theorems 5.3, 5.7]{HuLiuVertexIsoperimetrySignedGraph}.
Combining these two bounds together, they obtained \cite[Theorem 1.3]{HuLiuVertexIsoperimetrySignedGraph}, which provides the bound $\Omega (\frac{\vcc_\Gamma^2}d)$ on the lower spectral gap of non-bipartite Cayley sum graphs and non-bipartite vertex-transitive graphs. 

We remark that combining the result of Hu and Liu \cite[Theorem 1.2]{HuLiuVertexIsoperimetrySignedGraph} with \cref{Thm:IntroCayleySum}, \cref{Cor:Bdd} can be improved. Indeed, let $\Gamma$ be a graph as in 
\cref{Thm:IntroCayleySum}. By \cref{Thm:IntroCayleySum}, the bound $\beta_{\ver, \Gamma} = \Omega(\vcc_\Gamma)$ holds.
\cite[Equation (1.7)]{HuLiuVertexIsoperimetrySignedGraph} provides the bound $\Omega(\frac{\beta_{\ver, \Gamma}^2}{16d})$ 
for the lower spectral gap of $\Gamma$. 
These two bounds together yield the bound $\Omega(\frac{\vcc_\Gamma^2}{d})$ on the lower spectral gap of $\Gamma$. 
\end{remark}

\section{A trapping lemma}

We apply a technique of Fre\u{\i}man \cite{FreimanGroupsInverseProb} to obtain the following.

\begin{lemma}
\label{Lemma:IndexTwoSubgroup}
Let $V$ be a nonempty finite set. 
Let $\scrV$ be a nonempty subset of $V$. Assume that a finite group $\calG$ acts transitively on $V$ and no index two subgroup of $\calG$ acts transitively on $V$. 
Let $\xi, \zeta, \delta$ be non-negative real numbers satisfying the inequalities
\begin{equation}
\label{Eqn:XiZetaKappaBdd1}
\frac{1 - \xi}2
\leq
\frac{|\scrV| }{|V|}\leq \frac {1 + \zeta} 2, 
\delta >0, 
\delta < \frac {1 - \zeta}2, 
\delta < \frac {1- 3\xi}4.
\end{equation}
Suppose
$$
|\scrV \cap \tau(\scrV)| 
\in 
(\delta |\scrV| , (1- \delta)|\scrV|)
$$
holds for no $\tau\in \calG$. 
Then 
the subset $H_\delta$ of $\calG$ defined by 
$
H_\delta
 := 
\{
\tau\in \calG \,:\,\, |\scrV \cap  (\tau(\scrV))| \geq ( 1 - \delta) |\scrV|
\}$
forms a subgroup of $\calG$ of index two, and for some orbit $\calO$ of some element of $V$ under the action of $H_\delta$, 
we have 
$
|\scrV \cap \calO^c| \leq \sqrt {\frac{\delta (1+ \zeta) } 2 } \frac{|V|}2$. 
\end{lemma}

\begin{proof}
Let us assume that for no $\tau \in \calG$, 
the integer $|\scrV \cap \tau(\scrV)|$ lies in 
$
(\delta |\scrV| , (1- \delta)|\scrV|)
$. 
Let $t$ denote the size of the stabilizer of some element of $V$ under $\calG$. 
If $\calG = H_\delta $, then using 
an estimate from \cite[p. 9]{VertexTra},
we would obtain 
$
t |\scrV|^2 
= \sum_{g\in \calG} |\scrV\cap g\scrV| 
\geq |\calG| |\scrV| (1-\delta)$, 
which yields 
$
|\calG|  (1-\delta)
\leq 
t |\scrV|
\leq
t \frac {1 + \zeta} 2 |V| 
= 
|\calG| \frac{1 + \zeta} 2
$, 
contradicting 
$\delta < \frac {1 - \zeta}2$. 
Hence $H_\delta $ is a proper subset of $\calG$. 
Note that $H_\delta $ contains the identity element of $\calG$. Let $\tau_1, \tau_2$ be elements of $H_\delta $. By the triangle inequality, 
\begin{align*}
|\scrV \setminus (\tau_1(\tau_2(\scrV)))| 
& \leq |\scrV \setminus (\tau_1(\scrV)) | + |(\tau_1(\scrV)) \setminus (\tau_1(\tau_2(\scrV)))| \\
& =  |\scrV \setminus  (\tau_1(\scrV)) | + |\scrV  \setminus (\tau_2(\scrV)) |\\
& = |\scrV | -  |\scrV \cap (\tau_1(\scrV))  | + |\scrV | - |\scrV  \cap (\tau_2(\scrV)) |\\
& \leq 2|\scrV| - 2(1- \delta) |\scrV| .
\end{align*}
Using $\delta < \frac {1-3\xi}4 < \frac 13$, we obtain 
$
|\scrV \cap (\tau_1(\tau_2(\scrV)))| 
 = |\scrV | - |\scrV \setminus (\tau_1(\tau_2(\scrV)))| 
 \geq (1 - 2\delta) |\scrV|
 > \delta |\scrV|,
$
which shows that $H_\delta $ contains $\tau_1\tau_2$. So $H_\delta $ is a proper subgroup of $\calG$. 
Note that 
$
t |\scrV|^2 
= \sum_{g\in \calG} |\scrV\cap g\scrV| 
\leq 
|\scrV| |H_\delta | + \delta |\scrV| (|\calG|- |H_\delta |)
$,
which yields 
$
\frac {1 - \xi} 2 |\calG| 
\leq
|H_\delta | +  \delta (|\calG|- |H_\delta |)$. 
This implies that 
$
\frac{|H_\delta |}{|\calG|} 
 \geq  \frac {\frac {1 - \xi} 2 - \delta }{1- \delta} 
 > \frac 13,
$
where the final inequality follows since
$\delta < \frac {1- 3\xi}4$. 
Hence $H_\delta $ is a subgroup of $\calG$ of index two. 

Since $\calG$ acts transitively on $V$ and no index two subgroup of $\calG$ acts transitively on $V$, it follows that there are precisely two orbits $\calO_1, \calO_2$ of the action of $H_\delta $ on $V$. Using an estimate from \cite[p. 10]{VertexTra}, it follows that 
$
|\scrV \cap \calO_1|
|\scrV \cap \calO_2|
= 
\frac 1 {2t} 
\sum_{g\in H_\delta ^c}
|\scrV \cap g^\mo \scrV|
\leq 
\frac 1 {2t}\delta
|H_\delta | 
 |\scrV|
= 
\frac \delta {4}
|V|
 |\scrV|$. 
So, one of the inequalities 
$|\scrV \cap \calO_1| 
\leq 
\sqrt 
{\frac{\delta} 4 |V| |\scrV|},
|\scrV \cap \calO_2| 
\leq 
\sqrt 
{\frac{\delta } 4 |V| |\scrV|}
$
holds.  Hence for some orbit $\calO$ of some element of $V$ under the action of $H_\delta $, 
we have 
$
|\scrV \cap \calO^c| 
\leq 
\sqrt 
{\frac{\delta } 4 |V| |\scrV|}
=
\sqrt {\frac{\delta (1+ \zeta) } 2 } \frac{|V|}2$. 
\end{proof}

\begin{lemma}
\label{Prop}
Let $V, \scrV, \calG, \delta, \xi, \zeta$ be as in \cref{Lemma:IndexTwoSubgroup}. In addition to the conditions of \cref{Lemma:IndexTwoSubgroup}, 
assume that $\mu, \kappa$ are real numbers satisfying the inequalities
\begin{equation}
\label{Eqn:XiZetaKappaBdd2}
\mu >0, \kappa\geq 0, 
\frac 1 {2\mu } - \xi - \kappa > 0, 
\delta 
< \frac 2 {1 + \zeta} 
\left(
\frac 1 {2\mu } - \xi - \kappa
\right)^2.
\end{equation}
Then the following hold. 
\begin{enumerate}
\item 

Let $\rho_1, \ldots, \rho_d: V \to V$ be permutations of $V$ such that $|\rho_i(\scrV) \cap \scrV|\leq \kappa |V|$ holds for any $1\leq i \leq d$. 
Suppose for any index two subgroup $H$ of $\calG$, and for any $u, v\in V$ lying in the same $H$-orbit with $u = \rho_i(v)$ for some $1\leq i \leq d$, the inequalities 
$
|\rho_i(Hv) \cap Hu | \geq \frac{|V|}{2\mu}
,
|\rho_i(H^c v) \cap H^c u | \geq \frac{|V|}{2\mu}
$
hold. Then 
$V$ admits a partition into disjoint subsets $V_1, V_2$ of equal size such that $V_j \cap (\cup_{i=1}^d \rho_i(V_j)) = \emptyset$ for any $j$. 
\item 
Let $\scrP(V)$ denote the power set of $V$, and $\calN : \scrP(V) \to \scrP(V)$ be a map such that 
$|\calN(\scrV) \cap \scrV | 
\leq \kappa |V|$ holds. Then there exists an orbit $\calO$ of $H_\delta$ such that for any map $f:\calO\to \calO$ with 
$$\text{
$f(X) \subseteq \calN(X)$ for any $X\subseteq \calO$,}$$
the inequality 
$\frac {|f(\calO)\cap \calO|}{|V|}< \frac 1{2\mu}
$
holds. 
\end{enumerate}
\end{lemma}

\begin{proof}
By \cref{Lemma:IndexTwoSubgroup}, there is an orbit $\calO$ of $H_\delta$ such that $
|\scrV \cap \calO^c| \leq \sqrt {\frac{\delta (1+ \zeta) } 2 } \frac{|V|}2$. 
Note that for any map $\Phi: \calO \to V$ and $|\Phi(\scrV) \cap \scrV|\leq \kappa |V|$, we have 
\begin{align*}
|\Phi(\calO) \cap \calO |
& \leq
|\Phi(\calO \cap \scrV ) \cap \calO |
+ 
|\Phi(\calO \cap \scrV^c ) \cap \calO|\\
& = 
|\Phi(\calO \cap \scrV ) \cap \calO \cap \scrV  |
+ |\Phi(\calO \cap \scrV ) \cap \calO \cap \scrV^c  |
+ 
|\Phi(\calO \cap \scrV^c ) \cap \calO |\\
& \leq
|\Phi(\scrV) \cap \scrV  |
+ 2 |\calO \cap \scrV^c  | \\
& = 
|\Phi(\scrV) \cap \scrV  |
+ 2 
(
|\calO| - |\scrV| + |\calO ^c \cap \scrV |
) \\
& \leq 
\kappa |V| 
+ 2\left( \frac{|V|}2 - 
\frac {1 - \xi} 2 |V|  + 
\sqrt {\frac{\delta (1+ \zeta) } 2 } \frac{|V|}2
\right) \\
& \leq 
\left(
\xi + \kappa  +
\sqrt 
{\frac{\delta (1 + \zeta) } 2 }\right) |V| \\
& < \frac{|V|}{2\mu}.
\end{align*}

Suppose the conditions in part (1) hold. 
If $\calO \cap (\cup_{j=1}^d \rho_j(\calO))$ is nonempty, then there are elements $u, v\in \calO$ such that $u = \rho_i(v)$ for some $i$, and hence, 
$
|\rho_i(\calO) \cap \calO|
 = |\rho_i(Hv) \cap H u| \geq \frac{|V|}{2\mu}
$, 
contradicting $|\rho_i(\calO) \cap \calO| < \frac{|V|}{2\mu}$. 
If $\calO^c \cap (\cup_{j=1}^d \rho_j(\calO^c))$ is nonempty, then there are elements $u, v\in \calO^c$ such that $u = \rho_i(v)$ for some $i$, and hence, 
$
|\rho_i(\calO) \cap \calO|
 = |\rho_i(H^c v) \cap H^c u| \geq \frac{|V|}{2\mu}
$, 
contradicting $|\rho_i(\calO) \cap \calO| < \frac{|V|}{2\mu}$. 
Consequently, the sets $\calO \cap (\cup_{j=1}^d \rho_j(\calO))$, $\calO^c \cap (\cup_{j=1}^d \rho_j(\calO^c))$ are empty.

In part (2), the inequality 
$|f(\scrV) \cap \scrV | 
\leq |\calN(\scrV) \cap \scrV | 
\leq \kappa |V|$ holds. Hence, 
the inequality 
$\frac {|f(\calO)\cap \calO|}{|V|}< \frac 1{2\mu}
$
follows. 
\end{proof}

\section{The bipartiteness constants}

\subsection{Notations}

Let $G$ be a finite group, and $S$ be a subset of $G$. The Cayley graph $C(G, S)$ (resp. $C_\Sigma(G, S)$ is defined as the graph having $G$ as its set of vertices and for two elements $x, y \in G$, the vertex $y$  is adjacent to $x$ if $y  =xs$ (resp. $y  = x^\mo s$) for some $s\in S$. Further, if  $\sigma$ is a group automorphism of $G$, then the twisted Cayley graph $C(G, S)^\sigma$ (resp. $C_\Sigma(G, S)^\sigma$ is defined as the graph having $G$ as its set of vertices and for two elements $x, y \in G$, the vertex $y$  is adjacent to $x$ if $y  =\sigma(xs)$ (resp. $y  = \sigma(x^\mo s)$) for some $s\in S$. 
Note that taking $\sigma$ as the identity map, we obtain $C(G, S)^\sigma = C(G, S)$ and $C_\Sigma(G, S)^\sigma = C_\Sigma(G, S)$. Thus, the class of twisted Cayley graphs (resp. twisted Cayley sum graphs) contains the class of Cayley graphs (resp. Cayley sum graphs). 
The twisted Cayley graphs and twisted Cayley sum graphs were studied in \cite{CheegerTwisted}. 
We remark that the notion of twisted Cayley graphs is related to the notion of generalised Cayley graphs as defined by Maru\v{s}i\v{c}, Scapellato and Zagaglia Salvi \cite{MarusicScapellatoZagagliaSalviGeneralizedCayleyGraph}.

Let $(V, E)$ be an undirected, $d$-regular graph with $|V|\geq 2$. 
For a subset $S$ of the vertex of $(V, E)$, its set of neighbours is denoted by $\calN(S)$. The edge boundary $\partial_\edge (S)$ of $S$,
and the vertex boundary $\partial_\ver  (S)$ of $S$ are 
defined as 
$$
\partial_\edge (S) 
= E(S, V \setminus S), \quad \partial_\ver (S) = \calN(S) \setminus S.
$$
The edge Cheeger constant $\ecc_{(V, E)}$ of $(V, E)$ is defined as 
$$\ecc_{(V, E)} = \min_{S\subseteq V, S\neq \emptyset, |S|\leq |V|/2}
\frac{|\partial_\edge (S)|}{d|S|}.$$
For disjoint subsets $L, R$ of $V$ with $L \cup R\neq \emptyset$, define 
$$
\beta_\edge (L, R) 
= \frac
{|E(L, R^c)|+ |E(R, L^c)|}
{d|L\cup R|}
= 
\frac
{|E(L,L)| + |E(R, R)| + |\partial_\edge  (L\cup R)|}
{d|L\cup R|}.
$$
As introduced by 
Trevisan \cite{TrevisanMaxCutTheSmallestEigenVal}, the edge bipartiteness constant $\beta_{\edge, (V, E)}$ of $(V, E)$  
is defined as 
$$\beta_{\edge, (V, E)}
=
\min_{L \subseteq V, R \subseteq V, L\cap R = \emptyset, 
L \cup R \neq \emptyset}
\beta_\edge (L, R).
$$
Motivated by the work of Trevisan \cite{TrevisanMaxCutTheSmallestEigenVal}, Moorman, Ralli and Tetali \cite{CayleyBottomBipartite} introduced the vertex bipartiteness constant. 
For disjoint subsets $L, R$ of $V$ with $L \cup R\neq \emptyset$, define 
$$
\beta_\ver(L, R)
=
\frac
{|L\cap \calN(L)| + |R \cap \calN(R)| + | \partial_\ver (L \cup R) |} 
{|L\cup R|}.
$$
The vertex bipartiteness constant $\beta_{\ver, (V, E)}$ of $(V, E)$ is defined by 
$$
\beta_{\ver, (V, E)}
= 
\min_{L \subseteq V, R \subseteq V, L\cap R = \emptyset, 
L \cup R \neq \emptyset}
\beta_\ver (L, R).
$$

In the following, $(V, E)$ denotes a finite, undirected graph (which may contain multiple edges, and even multiple loops at certain vertices) of degree $d$. 
Let $\theta_1, \ldots, \theta_d: V\to V$ be permutations such that the vertices $v, \theta_i(v)$ are adjacent in $(V,  E)$ for any $v\in V$ and $1\leq i \leq d$, and that $ \calN(v)$ is equal to $\cup_{i=1}^d \{\theta_i(v)\}$ for any $v\in V$. 
The Birkhoff-von Neumann theorem \cite[Theorem 5.5]{vanLintWilsonCourseCombi} guarantees the  existence of such permutations. 
Assume that $|V|\geq 2$, and a finite group $\calG$ acts transitively on $V$ and no index two subgroup of $\calG$ acts transitively on the set $V$.

\subsection{Bounding the edge bipartiteness constant}

Let $\ecc$ denote the edge Cheeger constant of $(V, E)$, and $\beta_\edge$ denote the edge bipartiteness constant of $(V, E)$. 

\begin{theorem}
\label{Thm:betaEdgeEcc}
Assume that 
the graph $(V, E)$ is non-bipartite.
Also assume that for any $g\in \calG$, there exists a permutation $\varsigma_g: V \to V$ such that for any $A, B\subseteq V$, the sets $E(A, B), E(g(A), \varsigma_g (B))$ are of the same size. 
\begin{enumerate}
\item 
 If $\calG$ acts on $V$ through automorphisms of $(V, E)$, then 
$$
\ecc \leq 
\begin{cases}
40 \beta_\edge & \text{ if } d\beta_\edge < \frac  1{40}, \\
40 d\beta_\edge & \text{ if } d\beta_\edge \geq \frac  1{40}.
\end{cases}
$$
\item 
If for each $\theta_i, 1\leq i\leq d$ and $v\in V$, there is an automorphism or an anti-automorphism $\psi_{i,v}$ of the group $\calG$ such that 
$\theta_i(g\cdot v) =  \psi_{i,v}(g) \cdot \theta_i(v)$
holds for any $g\in \calG$, then 
$$
\ecc \leq 
\begin{cases}
90 \beta_\edge & \text{ if } d\beta_\edge < \frac  1{90}, \\
90 d\beta_\edge & \text{ if } d\beta_\edge \geq \frac  1{90}.
\end{cases}
$$
Moreover, if $\psi_{i, v}$ fixes any index two subgroup of $\calG$ for any $1\leq i \leq d$ and $v\in V$, then 
$$
\ecc \leq 
\begin{cases}
40 \beta_\edge & \text{ if } d\beta_\edge < \frac  1{40}, \\
40 d\beta_\edge & \text{ if } d\beta_\edge \geq \frac  1{40}.
\end{cases}
$$
\end{enumerate}
Let $\Gamma$ be a Cayley graph, or a Cayley sum graph, or a twisted Cayley graph, or a twisted Cayley sum graph, or a vertex-transitive graph. Denote the degree of $\Gamma$ by $d$. Assume that $\Gamma$ is undirected, non-bipartite. Then 
the inequality  
$
\beta_{\edge, \Gamma }
\geq
\frac
{\ecc_\Gamma}
{90 d}
$
holds. 
\end{theorem}

\begin{proof}
We set $\mu=1$ in part (1). In part (2), we set $\mu = 2$, and if $\psi_{i,v}$ fixes any index two subgroup of $\calG$, then we set $\mu = 1$. 
Note that 
$\frac{8\varepsilon}{1 - 2\varepsilon} < \min\left\{ 
\frac {1 - \varepsilon}2, 
\frac 14 - \frac {3\varepsilon}{2} , 
\frac {2\left( \frac 1{2\mu} - 3\varepsilon\right)^2 }{ 1 +\varepsilon} 
\right\}$
holds for $\varepsilon = \frac 1{40}$ (resp. $\varepsilon = \frac 1{90}$) if $\mu = 1$ (resp. $\mu = 2$). 

Let $L, R$ denote disjoint subsets of $V$ such that 
$\beta_\edge = \beta_\edge(L,R)$. 
We assume that $|L| \geq |R|$. 
Let $Y$ denote the subset 
$V\setminus (L \cup R \cup \calN(L\cup R))$ of $V$. 
If $|Y| \geq \frac{|V|}2$, then $|L \cup R| \leq \frac{|V|}{2}$, 
and hence 
$
\ecc  \varepsilon \leq \ecc \leq \frac{|\partial_\edge (L\cup R)|}{d|L \cup R|} \leq \beta_\edge$. 
If $\varepsilon |V| \leq |Y| < \frac{|V|}2$, then 
\begin{align*}
\ecc 
& \leq \frac{|\partial_\edge (Y)|}{d|Y|} \\
& = \frac{\langle T1_Y, 1_{V\setminus Y} \rangle}{d|Y|} \\
& = \frac{\langle T1_Y, 1_{\partial_\ver (Y)} \rangle}{d|Y|}  \\
& = \frac{\langle T1_{\partial_\ver (Y)}, 1_Y \rangle}{d|Y|}  \\
& = \frac{\langle T1_{\partial_\ver (L\cup R)}, 1_{Y\cup L \cup R} \rangle}{d|Y|}  \\
& =  \frac{|\partial_\edge  (L \cup R)|}{d|Y|} \\
& \leq \frac{d\beta_\edge |L \cup R|}{d\varepsilon |V|} \\
& \leq \frac{\beta_\edge}\varepsilon.
\end{align*}
Moreover, if $|L| \geq \frac {1 + \varepsilon}2 |V|$, then using
$
|L|
-
|L\cap \calN(L)|
=|L\setminus \calN(L)| = 
|L \cap (\cup \theta_i(L))^c |
=
|L \cap (\cap \theta_i(L^c))|
\leq 
|\theta_i(L^c)|
= |L^c|$, we obtain 
\begin{align*}
\beta_\edge 
& \geq \frac{|E(L,L)|}{d|L\cup R|}\\
& = \frac{\langle T1_L, 1_L\rangle }{d|L\cup R|}\\
& = \frac{\sum_{i=1}^d\langle P_{\theta_i} 1_L, 1_L\rangle }{d|L\cup R|}\\
& = \frac{\sum_{i=1}^d |\theta_i (L)\cap L| }{d|L\cup R|}\\
& = \frac{\sum_{i=1}^d (|L| - |L \cap \theta_i (L^c)|) }{d|L\cup R|}\\
& \geq \frac{\sum_{i=1}^d (|L| - |\theta_i (L^c)| )}{d|L\cup R|}\\
& = \frac{\sum_{i=1}^d (|L| - |L^c|) }{d|L\cup R|}\\
& = \frac{|L| - |L^c| }{|L\cup R|}\\
& \geq \varepsilon \\
& \geq \ecc \varepsilon.
\end{align*}
Further, 
if $d \beta_\edge \geq \varepsilon$ holds, we have 
$\ecc  \leq 1
\leq \frac{d\beta_\edge}\varepsilon$. 
Henceforth, we assume that 
$$|Y| < \varepsilon |V|,
|L| < \frac {1 + \varepsilon} 2|V|,
d \beta_\edge < \varepsilon.$$
Note that 
\begin{align*}
|V \setminus (L \cup R) | 
 = |\partial_\ver (L \cup R)| + |Y| 
 \leq   |\partial_\edge  (L \cup R)| + |Y| 
 \leq d\beta_\edge |L\cup R| + |Y|
 < (d\beta_\edge + \varepsilon)|V|.
\end{align*}
Using $|L|\geq |R|$, we obtain 
$|L| 
> 
\frac{1 - d\beta_\edge - \varepsilon}2|V|
> 
\frac{1 - 2\varepsilon}2|V| $.
Also note that 
$
|\theta_i(L)\cap L|
\leq 
|\calN(L) \cap L|
\leq 
|E(L, L)|
\leq d\beta_\edge |V|$. 
Set $\scrV = L, \xi = 2\varepsilon , 
\zeta = \varepsilon,
\kappa = d\beta_\edge, \delta = \frac{8\varepsilon}{1 - 2\varepsilon}$. 
Note that 
$\frac 1 {2\mu } - \xi - \kappa >0,
\delta > 0, 
\delta < \frac {1 - \zeta}2, 
\delta < \frac {1- 3\xi}4, 
\delta 
< \frac 2 {1 + \zeta} 
(
\frac 1 {2\mu } - \xi - \kappa
)^2
$
holds, i.e., Equations \eqref{Eqn:XiZetaKappaBdd1}, \eqref{Eqn:XiZetaKappaBdd2} hold.

Suppose the condition of part (1) holds. 
Let us assume that there exists an index two subgroup $H$ of $\calG$ such that for some orbit $\calO$ of $H$ and for any map $f:\calO\to \calO$ with 
$$\text{
$f(X) \subseteq \calN(X)$ for any $X\subseteq \calO$,}$$
the inequality 
$\frac {|f(\calO)\cap \calO|}{|V|}< \frac 1{2}
$
holds. 
Let $T$ denote the adjacency operator of $(V, E)$. 
Since $(V, E)$ is not bipartite, there are elements $u, v$ of $\calO^\dag$ such that $\langle T1_v, 1_u\rangle\neq 0 $ where $\calO^\dag$ is one of $\calO, \calO^c = V \setminus \calO$. Let $H^\dag$ denote $H$ (resp. $H^c = \calG \setminus H$) if $\calO^\dag  = \calO$ (resp. $\calO^\dag = \calO^c$). Note that $\calO = H^\dag u = H^\dag v$. 
Since $\calG$ acts on $V$ through automorphisms of $(V, E)$, 
for any $g\in H$, it follows that 
$\langle T1_{g(v)}, 1_{g(u)} \rangle =\langle T1_v, 1_u\rangle\neq 0 $, which implies that $\calO^\dag \subseteq \calN(\calO^\dag)$, and hence, $\calO \subseteq \calN(\calO)$. Note that for any $x, y\in \calO$, there is an element $h\in H$ such that $y = hx$, and hence $
\langle T1_\calO, 1_x \rangle 
= \langle T1_x, 1_\calO\rangle 
= \langle T P_h 1_x, P_h 1_\calO\rangle 
= \langle T 1_y, 1_\calO\rangle 
=\langle T1_\calO, 1_y \rangle 
$. Hence, by the Birkhoff-von Neumann theorem \cite[Theorem 5.5]{vanLintWilsonCourseCombi}, there exist permutations $\varrho_1, \ldots, \varrho_m$ of $\calO$ such that 
$T|_{\calO} = P_{\varrho_1} + \cdots + P_{\varrho_m}$, where $T|_{\calO}$ denotes the adjacency operator of the induced subgraph on $\calO$, i.e., the composite map $\ell^2(\calO) \to \ell^2(V) \xra{T}\ell^2(V) \to \ell^2(\calO)$, where the first map is the extension-by-zero map, and the final map is the orthogonal projection.
Taking $f=\varrho_1$, we obtain 
$|f(\calO)\cap \calO| =
|\varrho_1(\calO)\cap \calO|
= |\calO| = \frac{|V|}2$, which contradicts $|f(\calO)\cap \calO|< \frac{|V|}2$. 
Applying \cref{Prop}(2), it follows that 
$
|L \cap \tau(L)| 
\in 
(\delta |L| , (1- \delta)|L|)
$
holds for some $\tau \in \calG$. 

Suppose the condition of part (2) holds. 
Let $t$ denote the size of the stabilizer of some element of $V$ under $\calG$. Let $H$ be an index two subgroup of  $\calG$ and $u,v$ be elements of $V$ lying in the same $H$-orbit with $u = \theta_i(v)$ for some $1\leq i \leq d$. Then for $H^\dag = H, H^c$, we have 
\begin{align*}
|\theta_i(H^\dag v) \cap H^\dag u|
& = |\psi_{i,v} (H^\dag) \theta_i(v) \cap H^\dag \theta_i(v)|\\
& \geq |(\psi_{i,v} (H^\dag) \cap H^\dag) \theta_i(v)| \\
& \geq \frac{|\psi_{i,v} (H^\dag) \cap H^\dag|}{t} \\
& \geq \frac{|H|}{2t} \\
&= \frac{|V|}{4}
\end{align*}
and if $\psi_{i,v}$ fixes $H$, then 
$ |\theta_i(H^\dag v) \cap H^\dag u| \geq \frac{|V|}{2}$ holds. Since $(V, E)$ is non-bipartite, it follows from  \cref{Prop}(1) that 
$
|L \cap \tau(L)| 
\in 
(\delta |L| , (1- \delta)|L|)
$
holds for some $\tau \in \calG$.

For any subsets $A, B, C, D$ of $V$, note that 
\begin{align*}
&
((A\cap C) \cup (B \cap D))
\times  
((A\cap C) \cup (B \cap D))^c
\\
& =
((A\cap C) \cup (B \cap D))
\times  
((A^c\cup C^c) \cap (B^c \cup D^c))
\\
& =
\bigl(
(A\cap C) 
\times  
((A^c\cup C^c) \cap (B^c \cup D^c))
\bigr)
\sqcup 
\bigl(
(B \cap D)
\times  
((A^c\cup C^c) \cap (B^c \cup D^c))
\bigr)
\\
& \subseteq
\bigl(
(A\cap C) 
\times  
(B^c \cup D^c)
\bigr)
\sqcup 
\bigl(
(B \cap D)
\times  
(A^c\cup C^c)
\bigr)
\\
& =
\bigl(
((A\cap C) 
\times  
B^c)
\cup 
((A\cap C) 
\times  
D^c)
\bigr)
\sqcup 
\bigl(
((B \cap D)
\times  
A^c)
\cup 
((B \cap D)
\times  
C^c)
\bigr)
\\
& \subseteq 
\bigl(
(A
\times  
B^c)
\cup 
(C
\times  
D^c)
\bigr)
\cup 
\bigl(
(B
\times  
A^c)
\cup 
(D
\times  
C^c)
\bigr)\\
& \subseteq 
\bigl(
(A
\times  
B^c)
\cup 
(B
\times  
A^c)
\bigr)
\cup 
\bigl(
(C
\times  
D^c)
\cup 
(D
\times  
C^c)
\bigr)
\end{align*}
holds. 
It follows that 
\begin{align*}
& |E(
(L\cap \tau(L) ) \cup (R \cap \varsigma_\tau (R))
, 
((L\cap \tau(L) ) \cup (R \cap \varsigma_\tau (R)))^c
) |\\
& \leq 
|E(L, R^c) | + |E(R, L^c ) | + |E( \tau(L) , \varsigma_\tau (R)^c) |+|E(\varsigma_\tau (R),  \tau(L) ^c ) |\\
& = 
|E(L, R^c) | + |E(R, L^c ) | + |E( L, R^c) |+|E(R, L^c ) |,\\
& |E(
(L\cap \tau(R)) \cup (R \cap \varsigma_\tau (L))
, 
((L\cap \tau(R)) \cup (R \cap \varsigma_\tau (L)))^c
) |\\
& \leq 
|E(L, R^c) | + |E(R, L^c ) | + |E( \tau(R), \varsigma_\tau (L)^c) |+|E(\varsigma_\tau (L),  \tau(R)^c ) |\\
& = 
|E(L, R^c) | + |E(R, L^c ) | + |E(R, L^c) |+|E(L, R^c ) |.
\end{align*}
Since the sets 
$
(L\cap \tau(L) ) \cup (R \cap \varsigma_\tau (R))
,
(L\cap \tau(R)) \cup (R \cap \varsigma_\tau (L))
$ 
are disjoint, one of them has size at most $|V|/2$, and hence, 
\begin{align*}
&
\ecc
\min 
\{
d |(L\cap \tau(L) ) \cup (R \cap \varsigma_\tau (R))|, 
d |(L\cap \tau(R)) \cup (R \cap \varsigma_\tau (L))|
\} \\
& \leq 
|E(L, R^c) | + |E(R, L^c ) | + |E(R, L^c) |+|E(L, R^c ) | \\
& \leq 2 d\beta_\edge |L\cup R|\\
& \leq 2 d\beta_\edge |V|.
\end{align*}
Note that 
$$
|(L\cap \tau(L) ) \cup (R \cap \varsigma_\tau (R)) |
\geq 
|\tau(L) \cap L| 
\geq 
\delta |L| 
\geq 
\frac{\delta (1 - d\beta_\edge - \varepsilon)}2|V|
> \frac{\delta (1 - 2\varepsilon)}2 |V|
\geq 4\varepsilon |V| .$$
By the triangle inequality, we obtain 
$
|L \setminus \tau(L) | 
\leq 
|L \setminus (\tau(R))^c| + |(\tau(R))^c \setminus \tau(L)| $,
which yields 
\begin{align*}
|(L\cap \tau(R)) \cup (R \cap \varsigma_\tau (L))|
& \geq 
|L \cap \tau(R) |\\
& \geq 
|L \setminus \tau(L) | 
- 
|(\tau(R))^c \setminus \tau(L)| \\
& = 
|L \setminus \tau(L) | 
- 
|R^c \setminus L|\\
& = 
|L \setminus \tau(L) | 
- 
|V \setminus (L \cup R)|\\
& \geq \delta |L| - 
|V \setminus (L \cup R)|\\
& \geq \left(4\varepsilon - (d\beta_\edge + \varepsilon) \right)|V| \\
& \geq 2\varepsilon |V|.
\end{align*}
Consequently, we obtain  $2\ecc d\varepsilon |V| \leq 2d\beta_\edge |V|$, or equivalently, $\varepsilon \ecc \leq \beta_\edge$. This completes the proofs of the bounds stated in parts (1), (2). 

To prove the last part, we take 
$$
\varsigma_g 
= 
\begin{cases}
\text{the left multiplication map by } g  & \text{ if } \Gamma = C(G, S), \\
\text{the right multiplication map by }g^\mo & \text{ if } \Gamma = C_\Sigma(G, S),\\
\text{the left multiplication map by }\sigma(g)  & \text{ if } \Gamma = C(G, S)^\sigma, \\
\text{the right multiplication map by }\sigma^\mo (g^\mo) & \text{ if } \Gamma= C_\Sigma(G, S)^\sigma, \\
\text{the left multiplication map by } g  & \text{ if } \Gamma \text{ is a vertex-transitive graph}.
\end{cases}
$$
We claim that $\varsigma_g$ has the property that for any $A, B\subseteq V$, the sets $E(A, B), E(g(A), \varsigma_g (B))$ are of the same size. 
Indeed, if $\Gamma$ is a vertex-transitive  graph, then the sets $E(A, B), E(g(A), g (B))$ are of the same size for any $A, B\subseteq G$. If $\Gamma = C(G, S)^\sigma$, and if $b$ is adjacent to $a$ with $a, b\in G$, then $b = \sigma(as)$ for some $s\in S$, and hence for any $g\in G$, we have $\sigma(g)b = \sigma(gas)$, which  implies that $\sigma(g)b$ is adjacent to $ga$, and consequently, the sets $E(A, B), E(gA, B \sigma(g))$ are of the same size for any $A, B\subseteq G$. 
Further, if $\Gamma = C_\Sigma (G, S)^\sigma$, and if $b$ is adjacent to $a$ with $a, b\in G$, then $b = \sigma(a^\mo s)$ for some $s\in S$, and hence for any $g\in G$, we have $b \sigma^\mo(g^\mo) 
= \sigma(a^\mo s)\sigma^\mo(g^\mo)
= \sigma((ga)^\mo g s)\sigma^\mo(g^\mo)
= \sigma((ga)^\mo g s \sigma^\mo (\sigma^\mo(g^\mo)))
$. Since the graph $C_\Sigma (G, S)^\sigma$ is undirected, it follows that $S$ contains the element $g s \sigma^\mo (\sigma^\mo(g^\mo))$, which  implies that $b \sigma^\mo(g^\mo) $ is adjacent to $ga$, and consequently, the sets $E(A, B), E(gA,  B\sigma^\mo(g^\mo) )$ are of the same size for any $A, B\subseteq G$. This proves the claim. 

If $\Gamma$ is a vertex-transitive graph, then the condition of part (1) holds. 
Note that the vertex set of a twisted Cayley graph, and that  of a twisted Cayley sum graph carries an action of the underlying group via left multiplication. This action is transitive and no index two subgroup of the group acts transitively on the vertex set. 
If $\Gamma = C(G, S)^\sigma$, then writing $S = \{s_1, \ldots, s_d\}$, and taking $\theta_i(v) = \sigma(vs_i)$ and $\psi_{i,v}(g) = \sigma(g)$, it follows that 
$\theta_i(gv) = \sigma(gvs_i) 
= \sigma(g) \sigma(vs_i) = \psi_{i,v}(g) \theta_i(v)$. Further, if $\Gamma = C_\Sigma(G, S)^\sigma$, then writing $S = \{s_1, \ldots, s_d\}$, and taking $\theta_i(v) = \sigma(v^\mo s_i)$ and $\psi_{i,v}(g) = \sigma(v^\mo g^\mo v )$, it follows that 
$\theta_i(gv) = \sigma((gv)^\mo s_i) 
= \sigma(v^\mo g^\mo v ) \sigma(v^\mo s_i) = \psi_{i,v}(g ) \theta_i(v)$. So, if $\Gamma$ is a twisted Cayley graph or a twisted Cayley sum graph, then the condition of part (2) holds. 
Consequently, 
the inequality  
$
\beta_{\edge, \Gamma }
\geq
\frac
{\ecc_\Gamma}
{90 d}
$
holds if $\Gamma$ is a Cayley graph, or a Cayley sum graph, or a twisted Cayley graph, or a twisted Cayley sum graph, or a vertex-transitive graph. 
\end{proof}

\subsection{Bounding the vertex bipartiteness constant}

Let $\vcc$ denote the vertex Cheeger constant of $(V, E)$, and $\beta_\ver$ denote the vertex bipartiteness constant of $(V, E)$. 

\begin{theorem}
\label{Thm:betaEdgeVcc}
Assume that 
the graph $(V, E)$ is non-bipartite.
Also assume that for any $g\in \calG$, there exists a permutation $\varsigma_g: V \to V$ such that for any $A, B\subseteq V$, the sets $E(A, B), E(g(A), \varsigma_g (B))$ are of the same size. 
\begin{enumerate}
\item 
 If $\calG$ acts on $V$ through automorphisms of $(V, E)$, then 
$$
\vcc \leq 60 \beta_\ver.
$$
\item 
If for each $\theta_i, 1\leq i\leq d$ and $v\in V$, there is an automorphism or an anti-automorphism $\psi_{i,v}$ of the group $\calG$ such that 
$\theta_i(g\cdot v) =  \psi_{i,v}(g) \cdot \theta_i(v)$
holds for any $g\in \calG$, then 
$$
\vcc \leq 
135 \beta_\ver.
$$
Moreover, if $\psi_{i, v}$ fixes any index two subgroup of $\calG$ for any $1\leq i \leq d$ and $v\in V$, then 
$$
\vcc \leq 60 \beta_\ver.
$$
\end{enumerate}
Let $\Gamma$ be a Cayley graph, or a Cayley sum graph, or a twisted Cayley graph, or a twisted Cayley sum graph, or a vertex-transitive graph. Denote the degree of $\Gamma$ by $d$. Assume that $\Gamma$ is undirected, non-bipartite. Then 
the inequality  
$
\beta_{\ver, \Gamma }
\geq
\frac
{\vcc_\Gamma}
{135 }
$
holds. 
\end{theorem}

\begin{proof}
We set $\mu=1$ in part (1). In part (2), we set $\mu = 2$, and if $\psi_{i,v}$ fixes any index two subgroup of $\calG$, then we set $\mu = 1$. 
Note that 
$\frac{8\varepsilon}{1 - 2\varepsilon} < \min\left\{ 
\frac {1 - \varepsilon}2, 
\frac 14 - \frac {3\varepsilon}{2} , 
\frac {2\left( \frac 1{2\mu} - 3\varepsilon\right)^2 }{ 1 +\varepsilon} 
\right\}$
holds for $\varepsilon = \frac 1{40}$ (resp. $\varepsilon = \frac 1{90}$) if $\mu = 1$ (resp. $\mu = 2$). 

Let $L, R$ denote disjoint subsets of $V$ such that 
$\beta_\ver = \beta_\ver(L,R)$. 
We assume that $|L| \geq |R|$. 
Let $Y$ denote the subset 
$V\setminus (L \cup R \cup \calN(L\cup R))$ of $V$. 
If $|Y| \geq \frac{|V|}2$, then $|L \cup R| \leq \frac{|V|}{2}$, 
and hence 
$
\vcc  \leq \frac{|\partial_\ver (L\cup R)|}{|L \cup R|} \leq \beta_\ver$. 
If $\varepsilon |V| \leq |Y| < \frac{|V|}2$, then 
$$
\vcc  \leq \frac{|\partial_\ver (Y)|}{|Y|} =  \frac{|\partial_\ver (L \cup R)|}{|Y|}  \leq 
\frac{\beta_\ver |L \cup R|}{\varepsilon |V|} 
\leq \frac{\beta_\ver}\varepsilon.
$$
Moreover, if $|L| \geq \frac {1 + \varepsilon}2 |V|$, then 
$|L\setminus \calN(L)| = 
|L \cap (\cup \theta_i(L))^c |
=
|L \cap (\cap \theta_i(L^c))|
\leq 
|\theta_i(L^c)|
= |L^c|$, which implies that 
\begin{align*}
\beta_\ver 
 \geq \frac{|L\cap \calN(L)|}{|L\cup R|}
 \geq \frac{|L\cap \calN(L)|} {|V|}
 \geq \frac{|L| - |L^c|}{|V|}
 \geq \varepsilon
\geq \frac {2\varepsilon }3h  .
\end{align*}
Further, if $\beta_\ver \geq \varepsilon$ holds, we have 
$\vcc  \leq \vcc  \frac{\beta_\ver}\varepsilon
\leq \frac{\nu \beta_\ver}\varepsilon$. 
Henceforth, we assume that 
$$|Y| < \varepsilon |V|,
|L| < \frac {1 + \varepsilon} 2|V|,
 \beta_\ver < \varepsilon.$$
Note that 
\begin{align*}
|V \setminus (L \cup R) | 
 = |\partial_\ver (L \cup R)| + |Y| 
 \leq \beta_\ver |L\cup R| + |Y|
 < (\beta_\ver + \varepsilon)|V|.
\end{align*}
Using $|L|\geq |R|$, we obtain 
$|L| 
> 
\frac{1 - \beta_\ver - \varepsilon}2|V|
> 
\frac{1 - 2\varepsilon}2|V| $. 
Also note that $|\theta_i(L) \cap L| 
\leq |\calN(L) \cap L|
\leq
\beta_\ver |V|
$. 
Set $\scrV = L, \xi = 2\varepsilon , 
\zeta = \varepsilon,
\kappa = \beta_\ver, \delta = \frac{8\varepsilon}{1 - 2\varepsilon}$. 
Note that 
$\frac 1 {2\mu } - \xi - \kappa >0,
\delta > 0, 
\delta < \frac {1 - \zeta}2, 
\delta < \frac {1- 3\xi}4, 
\delta 
< \frac 2 {1 + \zeta} 
(
\frac 1 {2\mu } - \xi - \kappa
)^2
$
holds, i.e., Equations \eqref{Eqn:XiZetaKappaBdd1}, \eqref{Eqn:XiZetaKappaBdd2} hold.
Under the conditions of part (1), (2), it follows from the proof of \cref{Thm:betaEdgeEcc} that 
$
|L \cap \tau(L)| 
\in 
(\delta |L| , (1- \delta)|L|)
$
holds for some $\tau \in \calG$. 

For two disjoint subsets $A, B$ of $V$, note that
\begin{align*}
(A
\times  
B^c)
\cup 
(B
\times  
A^c)
& = 
(A \times A) \cup (A \times (A \cup B)^c) 
\cup 
(B \times B) \cup (B \times (A \cup B)^c)  \\
& = 
(A \times A) \cup (B \times B) \cup (A \cup B \times (A \cup B)^c) 
\end{align*}
holds. 
As in the proof of \cref{Thm:betaEdgeEcc}, for any subsets $A, B, C, D$ of $V$, 
\begin{align*}
((A\cap C) \cup (B \cap D))
\times  
((A\cap C) \cup (B \cap D))^c
 \subseteq 
\bigl(
(A
\times  
B^c)
\cup 
(B
\times  
A^c)
\bigr)
\cup 
\bigl(
(C
\times  
D^c)
\cup 
(D
\times  
C^c)
\bigr)
\end{align*}
holds. 
So, for any subsets $A, B, C, D$ of $V$ with $A\cap B = \emptyset$, we have 
\begin{align*}
&
((A\cap C) \cup (B \cap D))
\times  
((A\cap C) \cup (B \cap D))^c
\\
& \subseteq 
\bigl(
(A \times A) \cup (B \times B) \cup (A \cup B \times (A \cup B)^c)
\bigr)
\cup  
\bigl(
(C
\times  
D^c)
\cup 
(D
\times  
C^c)
\bigr).
\end{align*}
Consequently, for any subsets $A, B, C, D$ of $V$ with $A\cap B =C\cap D = \emptyset$, we have 
\begin{align*}
& |\partial_\ver ((A\cap \tau(C)) \cup (B \cap \varsigma_\tau (D)))|\\
& \leq 
|A\cap \calN(A)| + |B\cap \calN(B)|  + |\partial_\ver(A \cup B)| 
+ 
|\calN(\tau(C)) \cap \varsigma_\tau (D)^c | + |\calN(\varsigma_\tau (D)) \cap \tau(C)^c|\\
& \leq 
|A\cap \calN(A)| + |B\cap \calN(B)|  + |\partial_\ver(A \cup B)| 
+ 
|\calN(C) \cap D^c| + |\calN(D) \cap C^c|\\
& \leq 
|A\cap \calN(A)| + |B\cap \calN(B)|  + |\partial_\ver(A \cup B)|  \\
& + 
|\calN(C) \cap C | + |\calN(C) \cap (C\cup D)^c | + 
|\calN(D) \cap D | + |\calN(D) \cap (C\cup D)^c | \\
& \leq 
|A\cap \calN(A)| + |B\cap \calN(B)|  + |\partial_\ver(A \cup B)|  +
|\calN(C) \cap C | +
|\calN(D) \cap D | + 2|\partial_\ver(C \cup D)| .
\end{align*}
It follows that 
$$|\partial_\ver (
(L\cap \tau(L) ) \cup (R \cap \varsigma_\tau (R))
)|
\leq 
2|L\cap \calN(L)| + 2|R\cap \calN(R)| + 3|\partial_\ver(L \cup R)| ,$$
$$|\partial_\ver (
(L\cap \tau(R) ) \cup (R \cap \varsigma_\tau (L))
)|
\leq 
2|L\cap \calN(L)| + 2|R\cap \calN(R)| + 3|\partial_\ver(L \cup R)| .$$
Since the sets 
$
(L\cap \tau(L) ) \cup (R \cap \varsigma_\tau (R))
,
(L\cap \tau(R)) \cup (R \cap \varsigma_\tau (L))
$ 
are disjoint, one of them has size at most $|V|/2$, and hence, 
\begin{align*}
&
\vcc
\min 
\{
 |(L\cap \tau(L) ) \cup (R \cap \varsigma_\tau (R))|, 
 |(L\cap \tau(R)) \cup (R \cap \varsigma_\tau (L))|
\} \\
& \leq 
2|L\cap \calN(L)| + 2|R\cap \calN(R)| + 3|\partial_\ver(L \cup R)|  \\
& \leq 3 d\beta_\ver |L\cup R|\\
& \leq 3 d\beta_\ver |V|.
\end{align*}
As in the proof of \cref{Thm:betaEdgeEcc}, 
we have the bounds 
$$
|(L\cap \tau(L) ) \cup (R \cap \varsigma_\tau (R)) |
\geq 
\delta |L| 
\geq 4\varepsilon |V| ,$$
and 
\begin{align*}
|(L\cap \tau(R)) \cup (R \cap \varsigma_\tau (L))|
 \geq \delta |L| - 
|V \setminus (L \cup R)|
 \geq \left(4\varepsilon - (\beta_\ver + \varepsilon) \right)|V| 
 \geq 2\varepsilon |V|.
\end{align*}
Consequently, we obtain  $2\vcc \varepsilon |V| \leq 3\beta_\ver |V|$, or equivalently, $2\varepsilon \vcc \leq 3\beta_\ver$. This completes the proofs for the bounds stated in parts (1), (2). 

From the proof of 
\cref{Thm:betaEdgeEcc}, it follows that for any $A, B\subseteq V$, the sets $E(A, B)$, $E(g(A), \varsigma_g (B))$ are of the same size. 
The same proof also implies that if $\Gamma$ is a vertex-transitive graph, then the condition of part (1) holds, and if $\Gamma$ is a twisted Cayley graph or a twisted Cayley sum graph, then the condition of part (2) holds. 
Consequently, 
the inequality  
$
\beta_{\ver, \Gamma }
\geq
\frac
{\vcc_\Gamma}
{135 }
$
holds if $\Gamma$ is a Cayley graph, or a Cayley sum graph, or a twisted Cayley graph, or a twisted Cayley sum graph, or a vertex-transitive graph.
\end{proof}

We remark that the proofs of Theorems \ref{Thm:betaEdgeEcc}, \ref{Thm:betaEdgeVcc} rely on some counting arguments due to 
Moorman, Ralli and Tetali \cite{CayleyBottomBipartite}.

\section{The Cheeger constant of the square graph}
\label{Sec:CheegerConstantOfSquareGraph}
Let $V$ be a finite set. 
Let $\ell^2(V)$ denote the Hilbert space of functions $f: V \to \bbC$ equipped with the inner product 
$$\langle f, g\rangle : = \sum_{v\in V} f(v) \overline {g(v)}$$
and the norm 
$$||f||_{\ell^2(V)} 
: = 
\sqrt{\langle f, f\rangle}.$$
If $\rho: V \to V$ is a permutation, then it induces an operator $P_\rho:\ell^2(V) \to \ell^2(V)$ defined by 
$$(\rho\cdot f)(v) : = f(\rho^\mo v),$$
and in particular, $\rho \cdot 1_v =  1_{\rho(v)}$ since 
$(\rho \cdot 1_v) (u) 
= 
1_v (\rho^\mo u) 
= 1_{\rho(v) }(u)
$. 
An operator $P : \ell^2(V) \to \ell^2(V)$ is said to be a permutation operator if $P (\{1_v \,|\, v\in V\}) = \{1_v \,|\, v\in V\}$, i.e., $P = P_\rho$ for some permutation $\rho: V \to V$. 
Let $\calG$ be a finite group acting on $V$ from the left. The induced action of $\calG$ on $\ell^2(V)$ is defined by 
$$(g\cdot f)(v) : = f(g^\mo v).$$
Let $T: \ell^2(V) \to  \ell^2(V)$ be an  operator. 
For $F\subseteq V$, set 
$$
\vol_T (F) 
:= 
\langle T 1_V, 1_{F} \rangle
.$$
When $|V| \geq 2$ and $\vol_T(F)$ is a positive real number for any $\emptyset \neq F \subseteq V$, the edge Cheeger  constant of $T$, denoted by $\ecc_T$,  is defined by  
$$
\ecc_T 
:=
\min _{F \subseteq V, F \neq \emptyset, V }
\frac
{
\langle T 1_F, 1_{V\setminus F} \rangle
}
{
\min 
\{
\vol_T(F), \vol_T(V\setminus F)
\}
}.
$$

\begin{condition}
\label{Cond}
Let $V$ be a finite set with $|V|\geq 2$.
\begin{enumerate}
 \item 
 Let $T: \ell^2(V) \to \ell^2(V)$ be a self-adjoint operator. 
\item Suppose 
$
\langle 
T1_v, 1_u 
\rangle$ 
is a non-negative integer for any $u, v\in V$, and $T$ has degree $d >0$, i.e., $\vol_T(\{u\}) = d$ for any $u\in V$. 
\item  Let $\rho_1, \ldots, \rho_d: V\to V$ be permutations such that 
$T = P_{\rho_1} + \cdots + P_{\rho_d}$.

\item Assume that $-1$ is not an eigenvalue of $T$. 
\item 
Let $\calG$ be a group acting transitively on $V$, and assume that the action of $\calG$ on $\ell^2(V)$ commutes with the operator $T^2$, and that no index two subgroup of $\calG$ acts transitively on $V$. 
\end{enumerate}
\end{condition}

By the Birkhoff-von Neumann theorem \cite[Theorem 5.5]{vanLintWilsonCourseCombi}, under \cref{Cond}(2), there exist permutations $\rho_1, \ldots, \rho_d$ of $V$ such that 
$T = P_{\rho_1} + \cdots + P_{\rho_d}$.

\begin{theorem}
\label{Thm:Sec2Intro}
Suppose \cref{Cond} holds. 
The following statements hold. 
\begin{enumerate}
\item 
If the induced action of $\calG$ on $\ell^2(V)$ commutes with $T$, then 
$$
\ecc_{T^2}  
>
\frac 
{\ecc_{T}  ^{2}}{20d}
.
$$

\item If for each $1\leq i \leq d$ and $v\in V$, there is an automorphism or anti-automorphism $\psi_{i, v}$ of the group $\calG$ such that 
$P_{\rho_i} P_g (1_v)
= P_{\psi_{i, v}(g) } P_{\rho_i} (1_v)
$
holds for any $g\in \calG$, then 
$$
\ecc_{T^2}  
>
\frac 
{\ecc_{T}  ^{2}}{48 d}
.
$$
Further, if $\psi_{i, v}$ fixes any index two subgroup of $\calG$, then 
$$
\ecc_{T^2}  
>
\frac 
{\ecc_{T}  ^{2}}{20d}
.
$$
 \end{enumerate}
\end{theorem}
We will prove \cref{Thm:Sec2Intro} after obtaining \cref{Lemma:Dichotomy}. At the end of this section, we will establish \cref{Thm:EccEccBddAlgGraph} using \cref{Thm:Sec2Intro}. 
In the following, the edge Cheeger constant $\ecc_{T^2}$ of $T^2$ is denoted by $\psi$. Let $A$ be a subset of $V$ with $A\neq \emptyset, V$, satisfying $\vol_{T^2}(A) \leq \vol_{T^2}(V\setminus A)$ and 
$$\langle T^2 1_A, 1_{V\setminus A} \rangle = \psi \vol_{T^2}(A).$$
For a subset $X$ of $V$, define 
$$\calN(X) : = 
\{v\in V
\,|\,
\langle 
T1_v, 1_X
\rangle
\neq 0
\}.$$
Note that since $T$ is self-adjoint, we have 
$
\langle T 1_F, 1_{V\setminus F} \rangle
= 
\langle T 1_{V\setminus F} , 1_F\rangle
$
for any $F\subseteq V$, which implies 
\begin{equation}
\label{Eqn:LargeSetExpan}
\langle T 1_F, 1_{V\setminus F} \rangle
= 
\langle T 1_{V\setminus F} , 1_F\rangle
\geq \ecc_T \vol_T(V\setminus F)
\end{equation}
for any $F\subseteq V$ with $\vol_T(F) \geq \vol_T(V\setminus F)$. 

\begin{lemma}
\label{Lemma:BddXNXThroughX}
For any subset $X$ of $V$, 
\begin{align*}
\langle T 1_{X \cup \calN(X)}, 1_{V \setminus (X \cup \calN(X))} \rangle
\leq 
\langle T^2 1_X, 1_{V \setminus X} \rangle
\end{align*}
holds. 
\end{lemma}

\begin{proof}
Note that 
\begin{align*}
\langle T 1_{X \cup \calN(X)}, 1_{V \setminus (X \cup \calN(X))} \rangle
& 
\leq
\langle T (1_{X} + 1_{\calN(X)}), 1_{V \setminus (X \cup \calN(X))} \rangle\\
& 
\leq
\langle T (1_{X} + T 1_{X}), 1_{V \setminus (X \cup \calN(X))} \rangle\\
& 
=
\langle T 1_X, 1_{V \setminus (X \cup \calN(X))} \rangle
+
\langle T^2 1_X, 1_{V \setminus (X \cup \calN(X))} \rangle\\
& 
\leq
\langle T 1_X, 1_{V \setminus  \calN(X)} \rangle
+
\langle T^2 1_X, 1_{V \setminus X} \rangle\\
& 
=
\langle T^2 1_X, 1_{V \setminus X} \rangle.
\end{align*}
\end{proof}

\begin{lemma}
\label{Lemma:ANAIsLarge}
If $\psi < \frac{\ecc_T}{d}$, then for any $g  \in \calG$, the inequality 
$$
\vol_T(g  A \cup \calN(g  A)) 
\geq 
\vol_T (V \setminus (g  A \cup \calN(g  A)))$$
holds. 
\end{lemma}

\begin{proof}
Otherwise, using \cref{Lemma:BddXNXThroughX} and \cref{Cond}(2), (5), we would obtain 
\begin{align*}
 \ecc_T \vol_T(g  A) 
& \leq \ecc_T \vol_T(g  A \cup \calN(g  A)) \\
& \leq 
\langle T 1_{g  A \cup \calN(g  A)}, 1_{V \setminus (g  A \cup \calN(g  A))} \rangle\\
 & \leq \langle T^2 1_{g  A}, 1_{V \setminus g  A} \rangle\\
  & = \langle T^21_{A},  1_{V \setminus A} \rangle\\
& =  \psi \vol_{T^2}(A)\\
& = \psi d \vol_T(g  A) , 
\end{align*}
contradicting the hypothesis. 
\end{proof}

\begin{lemma}
\label{Lemma:AisHalfOfV}
If $\psi < \frac{\ecc_T}{d}$, then the inequalities 
$$
\vol_T  ( \calN(g  A) \cap g  A)  
\leq
2 \vol_T  (g  A) + 
\psi \vol_{T^2}(g  A)
-\vol_T(V) + \frac{\psi}{\ecc_T} \vol_{T^2} (g  A) 
$$
and 
$$
\vol_T(V) 
\leq
\left(
2 + d\psi + \frac{d\psi}{\ecc_T}
\right) 
\vol_T  (g  A) 
$$
hold for any $g \in \calG$. 
\end{lemma}

\begin{proof}
If $\psi < \frac{\ecc_T}{d}$, then using  \cref{Lemma:ANAIsLarge},  
Equation 
\eqref{Eqn:LargeSetExpan} and \cref{Lemma:BddXNXThroughX},
we obtain 
\begin{align*}
\ecc_T \vol_T (V \setminus (g  A \cup \calN(g  A)))
 \leq 
\langle T 1_{g  A \cup \calN(g  A)}, 1_{V \setminus (g  A \cup \calN(g  A))} \rangle
 \leq \langle T^2 1_{g  A}, 1_{V \setminus g  A} \rangle
 = \psi \vol_{T^2}(g  A),
\end{align*}
which implies that 
\begin{align*}
\vol_T(V) - \frac{\psi}{\ecc_T} \vol_{T^2} (g  A) 
& \leq 
\vol_T  (g  A \cup \calN(g  A)) \\
& = 
\langle 
T1_V , 1_{g  A \cup \calN(g  A)} \rangle \\
& = 
\langle 
T1_V , 1_{g  A} + 1_ { \calN(g  A)} - 1_{\calN(g  A) \cap g  A}\rangle \\
& = 
\vol_T  (g  A) - \vol_T  ( \calN(g  A) \cap g  A)  + 
\langle 
T1_V , 1_ { \calN(g  A) }\rangle \\
& = 
\vol_T  (g  A) - \vol_T  ( \calN(g  A) \cap g  A)  + 
\langle 
1_{V\setminus g  A}  , T 1_ { \calN(g  A) }\rangle 
+
\langle 
1_{g  A}, T 1_ { \calN(g  A) }\rangle \\
& \leq
\vol_T  (g  A) - \vol_T  ( \calN(g  A) \cap g  A)  + 
\langle 
1_{V\setminus g  A}  , T^2 1_{g  A}\rangle 
+
\langle 
T1_{g  A}, 1_V\rangle \\
& \leq
\vol_T  (g  A) - \vol_T  ( \calN(g  A) \cap g  A)  + 
\psi \vol_{T^2}(g  A)
+\vol_T  (g  A).
\end{align*}
The inequalities follow. 
\end{proof}

\begin{lemma}
\label{Lemma:PreDichotomy}
For any subsets $X, Y$ of $V$, 
\begin{align*}
& 
\langle 
T 1_X, 1_ X 
\rangle 
 + 
\langle 
 1_{X^c}, T1_{ X^c } 
\rangle 
 \leq 2\vol_T  ( \calN(X) \cap X)  
 + 
 \vol_T(V) 
 - 2 \vol_T(X) , \\
& \langle 
T 1_{X \Delta Y^c} 
, 1_{V\setminus 
(X \Delta Y^c )
}
\rangle  
  \leq \sum_{\calX= X, Y}
(
\langle 
T 1_\calX, 1_ \calX
\rangle 
 + 
\langle 
 1_{\calX^c}, T1_{\calX^c } 
\rangle 
)
\end{align*}
hold. 
Consequently, 
if $\psi < \frac{\ecc_T}{d}$, then the inequalities 
\begin{align*}
 \ecc_T 
\min \{
\vol_T(A \Delta (gA)^c) , 
\vol_T(V \setminus (A \Delta (gA)^c) )
\}
& \leq \langle 
T 1_{A \Delta (gA)^c} 
, 1_{V\setminus 
(A \Delta (gA)^c )
}
\rangle  \\
& \leq
\sum_{\calX= A, gA}
\frac{2\psi}{\ecc_T} (1 + \ecc_T) \vol_{T^2} (\calX)
\end{align*}
hold.
\end{lemma}

\begin{proof}
Note that 
\begin{align*}
\langle 
T 1_X, 1_ X 
\rangle 
 + 
\langle 
 1_{X^c}, T1_{ X^c } 
\rangle 
& = 
\langle 
T 1_X, 1_ X 
\rangle 
 + 
\langle 
T1_V - T1_X , 1_V - 1_X
\rangle  \\
& = 
2\langle 
T 1_X, 1_ X 
\rangle 
 + 
 \langle 
T 1_V, 1_V
\rangle 
-
\langle 
T 1_V, 1_ X 
\rangle 
-
\langle 
T 1_X, 1_ V
\rangle   \\
& = 
2\langle 
T 1_X, 1_ X 
\rangle 
 + 
 \vol_T(V) 
 - 2 \vol_T(X) \\
 & \leq 2\vol_T  ( \calN(X) \cap X)  
 + 
 \vol_T(V) 
 - 2 \vol_T(X) 
\end{align*}
and 
\begin{align*}
 \langle 
T 1_{X \Delta Y^c} 
, 1_{V\setminus 
(X \Delta Y^c )
}
\rangle  
& = \langle 
T 1_{X \Delta Y^c} 
, 1_{
X^c \Delta Y^c 
}
\rangle  \\
& = \sum_{i = 1}^d \langle 
P_{\rho_i} 1_{X \Delta Y^c} 
, 1_{
X^c \Delta Y^c 
}
\rangle  \\
& = \sum_{i = 1}^d \langle 
 1_{P_{\rho_i}(X)\Delta P_{\rho_i}(Y^c)} 
, 1_{
X^c \Delta Y^c 
}
\rangle  \\
& = \sum_{i = 1}^d 
|( P_{\rho_i}(X)\Delta P_{\rho_i}(Y^c))
\cap 
(
X^c \Delta Y^c 
)| \\
& \leq \sum_{i = 1}^d 
|( P_{\rho_i}(X)\Delta P_{\rho_i}(Y^c))
\Delta 
(
X^c \Delta Y
)| \\
& = \sum_{i = 1}^d 
| P_{\rho_i}(X)\Delta X^c \Delta P_{\rho_i}(Y^c)
\Delta 
Y | \\
& \leq \sum_{i = 1}^d 
(
| P_{\rho_i}(X)\Delta X^c | + |P_{\rho_i}(Y^c)
\Delta 
Y |
) \\
& = \sum_{i = 1}^d 
(
| P_{\rho_i}(X)\Delta X^c | + |P_{\rho_i}(Y)
\Delta 
Y^c |
) \\
& = \sum_{\calX= X, Y}\sum_{i = 1}^d 
(
| P_{\rho_i}(\calX)\cap \calX| + |\calX^c \cap P_{\rho_i}(\calX^c) | 
) \\
& = \sum_{\calX= X, Y}\sum_{i = 1}^d 
(
\langle 
1_{P_{\rho_i}(\calX)}, 1_ \calX 
\rangle 
 + 
\langle 
 1_{\calX^c}, 1_{ P_{\rho_i}(\calX^c) } 
\rangle 
)\\
& = \sum_{\calX= X, Y}
(
\langle 
T 1_\calX, 1_ \calX 
\rangle 
 + 
\langle 
 1_{\calX^c}, T1_{ \calX^c } 
\rangle 
)
\end{align*}
hold. 
Combining the above with \cref{Lemma:AisHalfOfV} and 
Equation 
\eqref{Eqn:LargeSetExpan}, we obtain 
\begin{align*}
& \ecc_T 
\min \{
\vol_T(A \Delta (gA)^c) , 
\vol_T(V \setminus (A \Delta (gA)^c) )
\}
\\
& \leq \langle 
T 1_{A \Delta (gA)^c} 
, 1_{V\setminus 
(A \Delta (gA)^c )
}
\rangle  \\
& \leq \sum_{\calX= A, gA}
(
2\vol_T  ( \calN(\calX) \cap \calX)  
 + 
 \vol_T(V) 
 - 2 \vol_T(\calX)
 )\\
& \leq \sum_{\calX= A, gA}
\left(
2\vol_T(\calX) - \vol_T(V)
+
2\psi \vol_{T^2}(\calX)
+ \frac{2\psi}{\ecc_T} \vol_{T^2} (\calX)
\right).
\end{align*}
\end{proof}

\begin{lemma}
\label{Lemma:Dichotomy}
Suppose $\psi < \frac{\ecc_T}{d}$. Then for a given $g\in G$,  one of the inequalities 
\begin{align}
\vol_T(A \cap gA)
& \geq 
 \vol_T(A) 
 \left( 
 1
 - 
\frac{2d\psi}{\ecc_T^2} (1 + \ecc_T)
\right), 
\label{Eqn:DichoLarge} \\
\vol_T(A \cap gA)
& \leq 
 \vol_T(A) 
 \left( 
\frac{2d\psi}{\ecc_T^2} (1 + \ecc_T)
\right)
\label{Eqn:DichoSmall}
\end{align}
  holds. 
\end{lemma}

\begin{proof}
Note that 
\begin{align*}
 \vol_T(V \setminus (A \Delta (gA)^c) )
& = \langle T1_V , 1_V - 1_{A \Delta (gA)^c}\rangle \\
& = \langle T1_V , 1_{(A \Delta (gA)^c)^c}\rangle \\
& = \langle T1_V , 1_{A \Delta gA}\rangle \\
& = \langle T1_V , 1_A + 1_{gA} - 21_{A \cap (gA)}\rangle \\
& = \vol_T(A) + \vol_T(gA) - 2\vol_T(A \cap gA).
\end{align*}
If 
$\vol_T(A \Delta (gA)^c) >  \vol_T (V \setminus (A \Delta (gA)^c))$, then 
using Equation 
\eqref{Eqn:LargeSetExpan} and \cref{Lemma:PreDichotomy} under the hypothesis $\psi < \frac{\ecc_T}{d}$, we obtain 
\begin{align*}
 \vol_T(A) + \vol_T(gA) - 2\vol_T(A \cap gA)
 =  \vol_T (V \setminus (A \Delta (gA)^c))
\leq \sum_{\calX= A, gA}
\frac{2\psi}{\ecc_T^2} (1 + \ecc_T) \vol_{T^2} (\calX), 
\end{align*}
which implies Equation \eqref{Eqn:DichoLarge}. 
Further, if  $\vol_T(A \Delta (gA)^c) \leq \vol_T (V \setminus (A \Delta (gA)^c))$, then applying \cref{Lemma:PreDichotomy}, we obtain 
\begin{align*}
 2\vol_T(A \cap gA)
& \leq \vol_T(V) 
- \vol_T(A) - \vol_T(gA) + 2\vol_T(A \cap gA)\\
& =  \vol_T(A \Delta (gA)^c) \\
& \leq
\sum_{\calX= A, gA}
\frac{2\psi}{\ecc_T^2} (1 + \ecc_T) \vol_{T^2} (\calX).
\end{align*}
This completes the proof. 
\end{proof} 

\begin{proof}
[Proof of \cref{Thm:Sec2Intro}]
We set $\mu=1$ in part (1). In part (2), we set $\mu = 2$, and if $\psi_{i,v}$ fixes any index two subgroup of $\calG$, then we set $\mu = 1$. 
Set 
 $$\scrV = A, x = \frac{d\psi}{\ecc_T}(1+\ecc_T),\xi = \frac {x}{2+x}, \zeta = 0, \delta = \frac{2x}{\ecc_T}, \kappa = \frac x2.$$
 Note that when $\mu =1$ (resp. $\mu =2$), it suffices to consider the case $x\leq \frac 1{10}$ (resp. $x\leq \frac 1{24}$), which we assume from now on. Note that $\psi < \frac{\ecc_T}d$ holds. 
We claim that $\delta \geq   \frac{1-3\xi}4$  (resp. $\delta \geq \frac 2 {1 + \zeta} 
(
\frac 1 {2\mu } - \xi - \kappa
)^2
$) when $\mu=1$ (resp. $\mu=2$). 
On the contrary, let us assume  that
$\delta <  \frac{1-3\xi}4$  (resp. $\delta < \frac 2 {1 + \zeta} 
(
\frac 1 {2\mu } - \xi - \kappa
)^2
$) when $\mu=1$ (resp. $\mu=2$). 
Under this assumption, Equations \eqref{Eqn:XiZetaKappaBdd1}, \eqref{Eqn:XiZetaKappaBdd2} hold.

To prove part (1), assume that the induced action of $\calG$ on $\ell^2(V)$ commutes with $T$. 
Since $\delta < \frac 12$ and $\psi < \frac{\ecc_T}d$, by \cref{Lemma:Dichotomy}, 
$
|\scrV \cap \tau(\scrV)| 
\in 
(\delta |\scrV| , (1- \delta)|\scrV|)
$
holds for no $\tau\in \calG$. 
Also note that $|\calN(\scrV) \cap \scrV| \leq \kappa |V|$. By \cref{Prop}(2), there exists an orbit $\calO$ of the subgroup $H_\delta$ such that for any map $f:\calO\to \calO$ with 
$$\text{
$f(X) \subseteq \calN(X)$ for any $X\subseteq \calO$,}$$
the inequality 
$\frac {|f(\calO)\cap \calO|}{|V|}< \frac 1{2}
$
holds. 
By \cref{Cond}(4), it follows that 
there are elements $u, v$ of $\calO^\dag$ such that $\langle T1_v, 1_u\rangle\neq 0 $ where $\calO^\dag$ is one of $\calO, \calO^c = V \setminus \calO$. 
Under the hypothesis of \cref{Thm:Sec2Intro}(1),
it follows that for some map $f:\calO\to \calO$ with 
$$\text{
$f(X) \subseteq \calN(X)$ for any $X\subseteq \calO$,}$$
we have 
$|f(\calO)\cap \calO| = \frac{|V|}2
$, as in the proof of  \cref{Thm:betaEdgeEcc}. This contradicts $|f(\calO)\cap \calO|< \frac{|V|}2$.

To prove part (2), let $H$ denote an index two subgroup of  $\calG$ and assume that $u,v$ are elements of $V$ lying in the same $H$-orbit with $u = \rho_i(v)$ for some $1\leq i \leq d$. Under the hypothesis of \cref{Thm:Sec2Intro}(2), for $H^\dag = H, H^c$, it follows that $|\rho_i (H^\dag v) \cap H^\dag u| \geq \frac{|V|}{2\mu}$, as in the proof of  \cref{Thm:betaEdgeEcc}. 
Since $\delta < \frac 12$ and $\psi < \frac{\ecc_T}d$, by \cref{Lemma:Dichotomy}, 
$
|\scrV \cap \tau(\scrV)| 
\in 
(\delta |\scrV| , (1- \delta)|\scrV|)
$
holds for no $\tau\in \calG$. 
By \cref{Prop}(1), $V$ admits a partition into disjoint subsets $V_1, V_2$ of equal size such that $V_j \cap (\cup_{i=1}^d \rho_i(V_j)) = \emptyset$ for any $j$. This contradicts \cref{Cond}(4). 

This proves the claim that $\delta \geq   \frac{1-3\xi}4$  (resp. $\delta \geq \frac 2 {1 + \zeta} 
(
\frac 1 {2\mu } - \xi - \kappa
)^2
$) when $\mu=1$ (resp. $\mu=2$). 
From $\delta \geq \frac {1 - 3\xi}4$, we obtain 
$\ecc_T \leq \frac{4x(x+2)}{1-x} 
\leq \frac{4x(\frac 1{10}+2)}{1-\frac 1 {10}} 
< 10x
$, which yields $\psi > \frac{\ecc_T^2}{10d(1 + \ecc_T)  }\geq \frac{\ecc_T^2}{20d}$.  
From 
$\delta \geq \frac 2 {1 + \zeta} 
(
\frac 1 {2\mu } - \xi - \kappa
)^2
$ with $\mu =2$, we obtain  
$\ecc_T \leq \frac{16x(x+2)^2}{(2x^2+7x-2)^2}
\leq \frac
{16x(\frac 1{24} +2)^2}
{(2 - \frac 2{24^2} - \frac 7{24})^2}
<24x
$, which yields $\psi > \frac{\ecc_T^2}{24 d(1 + \ecc_T)  }\geq \frac{\ecc_T^2}{48d}$.  
\end{proof}

We remark that the proof of \cref{Thm:Sec2Intro} relies on certain counting arguments due to Breuillard, Green, Guralnick and Tao \cite[Appendix E]{BGGTExpansionSimpleLie}.

\begin{proof}
[Proof of \cref{Thm:EccEccBddAlgGraph}]
Note that \cref{Cond} holds for vertex-transitive graphs.  So, 
\cref{Thm:Sec2Intro} implies 
\cref{Thm:EccEccBddAlgGraph} for vertex-transitive graphs. To prove \cref{Thm:EccEccBddAlgGraph}, it suffices to show that it holds for  twisted Cayley graphs and twisted Cayley sum graphs. Note that \cref{Cond}(2), (3) hold for these graphs. Since they are assumed to be undirected and non-bipartite, \cref{Cond}(1), (4) also hold. Also note that the vertex set of a twisted Cayley graph, and that  of a twisted Cayley sum graph carries an action of the underlying group via left multiplication. This action is transitive and no index two subgroup of the group acts transitively on the vertex set.

Let us consider a twisted Cayley graph $C(G, S)^\sigma$. Let $A$ denote the sum of the permutation operators on $\ell^2(G)$ corresponding to the right multiplication maps by the elements of   $S$. Let $P$ be a permutation operator on $\ell^2(G)$ such that $PA$ is self-adjoint. 
For $g\in G$, let $P_g$ denote the permutation operator on $\ell^2(G)$ corresponding to the left multiplication map by $g$.
Note that 
$(PA)^2 = 
(PA)(PA) 
=(PA)^t (PA) 
=A^t P^t PA
=A^t A $, 
and hence, for any $g\in G$, 
$
P_g (PA)^2 P_g^\mo
=P_g A^t A P_g^\mo
= A^t A = (PA)^2$ hold. 
So, the permutation operators on $\ell^2(G)$, corresponding to the left multiplication map by the elements of $G$, commute with the square of the adjacency operator of $C(G, S)^\sigma$, i.e., \cref{Cond}(5) holds for $C(G, S)^\sigma$.

Now, let us consider a twisted Cayley sum graph $C_\Sigma(G, S)^\sigma$. 
Let $A$ denote the sum of the permutation operators on $\ell^2(G)$ corresponding to the right multiplication maps by the elements of   $S$. 
Let $\calI$ denote the permutation operator on $\ell^2(G)$ corresponding to the inversion map on $G$. 
Let $P$ be a permutation operator on $\ell^2(G)$ such that $PA\calI$ is self-adjoint. 
For $g\in G$, let $P_g$ denote the permutation operator on $\ell^2(G)$ corresponding to the left multiplication map by $g$.
Note that 
$(PA\calI )^2 = 
(PA\calI )(PA\calI ) 
=(PA\calI )^t (PA\calI ) 
=\calI ^t A^t P^t PA \calI 
=\calI ^t A^t A \calI $.
For any $g \in G$, 
\begin{align*}
(P_g (PA\calI )^2 P_g^\mo)(1_v)
& = (P_g  \calI ^t A^t A \calI P_g^\mo)(1_v)\\
& = (P_g  \calI ^t A^t A \calI )(1_{g^\mo v})\\
& = (P_g  \calI ^t A^t A) (1_{v^\mo g})\\
& = (P_g  \calI ^t )(\sum_{s, t \in S}1_{v^\mo g s t^\mo})\\
& = P_g   (\sum_{s, t \in S}1_{ts^\mo g^\mo v})\\
& = \sum_{s, t \in S}1_{g ts^\mo g^\mo v}\\
& = (\sum_{s, t \in S}P_{g ts^\mo g^\mo} )(1_ v)
\end{align*}
hold for any $v\in G$, and hence, $P_g (PA\calI )^2 P_g^\mo = \sum_{s, t \in S}P_{g ts^\mo g^\mo}$, and $(PA\calI )^2 = \sum_{s, t \in S}P_{ts^\mo }$.

Let $x,y$ be elements of $G$. Suppose $y$ is adjacent to $x$, i.e., $y = \sigma(x^\mo s)$ holds for some $s\in S$. Since $C_\Sigma(G, S)^\sigma$ is undirected, for some $t\in S$, $x = \sigma(y^\mo t)$ holds, i.e., $\sigma(x^\mo s) \sigma^\mo (x)$ lies in $S$ for any $x\in G, s\in S$. Taking $x=e$, it follows that $S = \sigma(S)$. Hence, $\sigma(x^\mo ) s \sigma^\mo (x)$ lies in $S$ for any $x\in G, s\in S$. Taking $x = \sigma^\mo (g^\mo)$, it follows that $g s \sigma^\mo (\sigma^\mo (g^\mo))$ lies in $S$ for any $g\in G, s\in S$, i.e., $g S \sigma^\mo (\sigma^\mo (g^\mo)) = S$. Consequently, 
\begin{align*}
P_g (PA\calI )^2 P_g^\mo 
& = \sum_{s, t \in S}P_{g ts^\mo g^\mo}\\
& = \sum_{s, t \in S}P_{g t  \sigma^\mo (\sigma^\mo (g^\mo))  \sigma^\mo (\sigma^\mo (g)) s^\mo g^\mo}\\
& = \sum_{s, t \in S}P_{g t  \sigma^\mo (\sigma^\mo (g^\mo))  (gs \sigma^\mo (\sigma^\mo (g^\mo)) )^\mo}\\
&= \sum_{s, t \in S}P_{ts^\mo}\\
&= (PA\calI )^2 .
\end{align*}
So, the permutation operators on $\ell^2(G)$, corresponding to the left multiplication map by the elements of $G$, commute with the square of the adjacency operator of $C_\Sigma(G, S)^\sigma$, i.e., \cref{Cond}(5) holds for $C_\Sigma(G, S)^\sigma$. 
\end{proof}

\section{Acknowledgements}
The author acknowledges the INSPIRE Faculty Award (IFA18-MA123) from the Department of Science and Technology, Government of India.

\end{document}